\documentclass[a4paper,10pt]{article}
\usepackage{mathtext}
\usepackage[T1,T2A]{fontenc}
\usepackage[cp1251]{inputenc}
\usepackage[english]{babel}
\usepackage{amsmath}
\usepackage{amsfonts}
\usepackage{amssymb}
\usepackage{mathrsfs}
\usepackage{amsthm}
\usepackage{enumerate}
\usepackage{graphicx}

\usepackage{color}
\usepackage{euscript}

\usepackage{cite}

\newtheorem{Le}{Lemma}[section]
\newtheorem{Def}[Le]{Definition}

\newtheorem{Th}{Theorem}[section]
\newtheorem{Cor}[Le]{Corollary}
\newtheorem{Rem}[Le]{Remark}
\newtheorem{Conj}[Le]{Conjecture}
\newtheorem{Que}[Le]{Question}
\newtheorem{Ex}[Le]{Example}
\numberwithin{equation}{section}

\newcommand{\R}{\mathbb{R}}

\newcommand{\N}{\mathbb{N}}

\newcommand{\BB}{\mathbb{T}}
\newcommand{\F}{\mathcal{F}}
\newcommand{\AF}{\mathcal{AF}}

\DeclareMathOperator{\J}{J}

\newcommand{\eps}{\varepsilon}

\newcommand{\Lin}{\mathrm{T}}

\newcommand{\eq}[1]{\begin{equation}{#1}\end{equation}}
\newcommand{\mlt}[1]{\begin{multline}{#1}\end{multline}}
\newcommand{\alg}[1]{\begin{align}{#1}\end{align}}

\newcommand{\Set}[2]{\Big\{{#1}\,\Big|\;{#2}\Big\}}

\newcommand{\GeqrefTwo}[2]{\stackrel{\scriptscriptstyle{\eqref{#1},\eqref{#2}}}{\geq}}
\newcommand{\Gref}[1]{\stackrel{#1}{\geq}}
\newcommand{\Geqref}[1]{\stackrel{\scriptscriptstyle{\eqref{#1}}}{\geq}}

\DeclareMathOperator{\I}{I}

\newcommand{\E}{\mathbb{E}}

\newcommand{\Min}{\mathcal{M}_p}

\newcommand{\Bell}{\mathbb{B}}
\title{On~$\Phi$-inequalities for martingale fractional integration and their Bellman functions}
\author{Dmitriy Stolyarov\thanks{Supported by Russian Science Foundation grant N 19-71-10023}}
\begin{document}
\maketitle
\begin{abstract}
Inspired by a conjecture of Vladimir Maz'ya on~$\Phi$-inequalities in the spirit of Bourgain and Brezis, we establish some~$\Phi$-inequalities for fractional martingale transforms. These inequalities may be thought of as martingale models of~$\Phi$-inequalities for differential operators. The proofs rest on new simple Bellman functions.
\end{abstract}

\section{Introduction}
In~\cite{Mazya2010}, Vladimir Maz'ya conjectured that the inequality
\eq{\label{MazjaInequality}
\Big|\int\limits_{\R^d}\Phi(\nabla f(x))\,dx\Big| \lesssim \|\Delta f\|_{L_1(\R^d)}^{\frac{d}{d-1}},\qquad f\in C_0^\infty(\R^d),
}
holds true whenever~$\Phi \colon \R^d \to \R$ is a locally Lipschitz positively~$\frac{d}{d-1}$-homogeneous function satisfying
\eq{\label{MazjaCancellation}
\int\limits_{S^{d-1}} \Phi(\zeta)\,d\sigma(\zeta) = 0.
}
In the latter condition,~$\sigma$ stands for the~$(d-1)$-dimensional Hausdorff measure on the unit sphere. By the positive~$p$-homogeneity of a function~$\Psi\colon \R^d \to \R$ we mean the validity of the identity~$\Psi(t x) = t^p \Psi(x)$ for any~$x\in \R^d$ and any~$t > 0$. The notation~$A \lesssim B$ means that there exists a constant~$C$ such that~$A \leq CB$ for any choice of a parameter. For example, the constant in the inequality~\eqref{MazjaInequality} should not depend on the particular choice of~$f$. The necessity of condition~\eqref{MazjaCancellation} for~\eqref{MazjaInequality} may be verified by the example of~$f$ such that~$\Delta f$ mimics a delta measure. The conjecture was also listed as Problem~$5.1$ in~\cite{Mazya2018}. Some particular cases had been considered in~\cite{MazyaShaposhnikova2009}.

This inequality expresses a phenomenon that has recently attracted a lot of attention. In a broad sense, it says that though the Hardy--Littlewood--Sobolev inequality is not valid at the endpoint case~$p=1$, if one excludes delta measures from consideration in an appropriate manner, then the inequality becomes valid. Sometimes the resulting inequalities are called Bourgain--Brezis inequalities. We refer the reader to the papers~\cite{BourgainBrezis2002, BourgainBrezis2004, BourgainBrezis2007, BousquetVanSchaftingen2014, HernandezSpector2020, KMS2015, LanzaniStein2005, Mazya2007, Raita2018, Raita2019, SpectorVanSchaftingen2019, Spector2020, Stolyarov2020, Stolyarov2021, VanSchaftingen2004one, VanSchaftingen2004, VanSchaftingen2008, VanSchaftingen2010, VanSchaftingen2013} among many others and to the surveys~\cite{VanSchaftingen2014, Spector2019}.

The papers~\cite{ASW2018} and~\cite{Stolyarov2019} suggest a martingale interpretation of the phenomenon. While being interesting in itself, the martingale model develops intuition in this circle of questions. See~\cite{Stolyarov2020} and~\cite{Stolyarov2021} for the application of martingale techniques to the original Bourgain--Brezis inequalities for differential operators. The present paper provides a martingale version of the problem stated at the very beginning. We will formulate it in Section~\ref{s2} below and then provide solution in Sections~\ref{s3},~\ref{s4}, and~\ref{S5}; the last Section~\ref{s6} contains some comments on the original Maz'ya's problem. The solution rests upon a version of induction on scales called the Bellman function (or Burkholder) method. Though this is quite well-known and developed technique for proving inequalities in probability and harmonic analysis, it seems that it had never been applied to Bourgain--Brezis inequalities. The Bellman function arising in our reasonings seems to be interesting in itself (though we will not find the sharp Bellman function, only provide a supersolution in Theorem~\ref{Supersolution}). Very little is known about Bellman functions related to fractional operators, however, see~\cite{BanuelosOsekowski2017} and~\cite{Osekowski2014} for two sharp supersolutions.

The paper is self-contained, however, we refer the reader to the foundational papers~\cite{Burkholder1984},~\cite{NazarovTreil1996} and the books~\cite{Osekowski2012}, \cite{VasyuninVolberg2020} for the basics of the Bellman function method. 

I wish to thank Vladimir Maz'ya for attracting my attention to his problem and to Ilya Zlotnikov for reading the paper and improving the presentation.  

\section{Definition and setting}\label{s2}
 We will be using the notation from~\cite{ASW2018} and~\cite{Stolyarov2019} with some modifications; the idea of using the model described below goes back to~\cite{ChaoJanson1981} and~\cite{Janson1977}. Let~$m \geq 2$ and~$\ell$ be natural numbers. Consider the linear space
\eq{
V = \Set{x\in \R^m}{\sum_{j=1}^m x_j=0}
}
and a linear mapping~$\Lin\colon V \to V\otimes \R^\ell$. 

Let~$\F = \{\F_n\}_n$ be an~$m$-uniform filtration on a non-atomic probability space. By this we mean that each~$\F_n$ is a set algebra consisting of measurable sets, and any atom (minimal by inclusion non-empty element) of~$\F_n$ is split into~$m$ atoms in~$\F_{n+1}$ having equal probability. We also require~$\F_0$ to be the trivial algebra. The symbol~$\AF_n$ denotes the set of all atoms in~$\F_n$. 
For each~$\omega \in \AF_n$, we fix a map
\begin{equation*}
\J_\omega \colon [1\,..\,m] \to \{\omega' \in \AF_{n+1}\mid \omega' \subset \omega\}.
\end{equation*}
This fixes the tree structure on the set of all atoms. 

We will be considering~$\R$ or~$\mathbb{R}^\ell$-valued martingales adapted to~$\F$. Let~$F = \{F_n\}_n$ be an~$\mathbb{R}$-valued (or~$\R^\ell$-valued) martingale, i.e. a sequence of random variables that satisfy two requirements: first, each~$F_n$ is~$\F_n$-measurable (constant on the elements of~$\AF_n$) and second,~$\E(F_{n+1}\mid \F_n) = F_{n+1}$ for every~$n$. Define the martingale difference sequence by the rule
\begin{equation*}
f_{n+1} = F_{n+1} - F_n,\quad n \geq 0.
\end{equation*}
Now fix an atom~$\omega \in \AF_n$. The map~$\J_\omega$ may be naturally extended to the linear map that identifies an element of~$V$ with the restriction~$f_{n+1}|_{\omega}$ of a martingale difference to~$\omega$ (similarly, martingale differences of~$\R^\ell$-valued martingale correspond to elements of~$V\otimes \mathbb{R}^\ell$). In other words, the natural extension of~$\J_\omega$ identifies~$V$ with the space of real valued~$\F_{n+1}$-measurable functions on~$\omega$ having mean value zero. The said extension will be also denoted by~$\J_\omega$. 

We introduce the martingale analog of the Riesz potential:
\begin{equation*}
\I_\alpha[F] = \Big\{\sum\limits_{k=0}^{n}m^{-\alpha k}f_{k}\Big\}_{n}, \quad \alpha > 0.
\end{equation*}
Consider a vectorial modification of the Riesz potential:
\begin{equation}\label{ModifiedRiesz}
\BB_\alpha[F] = \bigg\{\sum\limits_{k=0}^{n-1} m^{-\alpha (k+1)}\sum\limits_{\omega \in \AF_k}\J_\omega\Big[\Lin\big[\J^{-1}_\omega [f_{k+1}|_{\omega}]\big]\Big]\bigg\}_n
\end{equation} 
There is a small inaccuracy in the notation here. Namely, the image of~$\J_\omega$ is formally defined as a function on~$\omega$. We have extended it by zero to the remaining part of the probability space. We will often denote by~$\BB_\alpha[F]$ not only the martingale~\eqref{ModifiedRiesz}, but also the series
\eq{
\sum\limits_{k=0}^\infty m^{-\alpha (k+1)}\sum\limits_{\omega \in \AF_k}\J_\omega\Big[\Lin\big[\J^{-1}_\omega [f_{k+1}|_{\omega}]\big]\Big].
}
We use the operator~$\J_\omega$ to make the formula mathematically correct. Informally, we may write it as
\eq{
\BB_\alpha[F] = \sum\limits_{k=0}^\infty m^{-\alpha (k+1)}\sum\limits_{\omega \in \AF_k}\Lin[f_{k+1}|_{\omega}].
} 

\begin{Ex}\label{Example1}
Let~$m=3$ and~$\ell=1$. Consider the operator~$\Lin\colon V \to V$ given by
\eq{\label{ExampleOfT}
\Lin[(x,y,z)] = (z-y, x-z,y-x),\qquad (x,y,z)\in V.
} 
For an atom~$\omega \in \AF_n$\textup, let~$\omega_1$\textup,~$\omega_2$\textup, and~$\omega_3$ be its kids in~$\AF_{n+1}$. Then\textup,
\mlt{\label{ExampleTalpha}
\BB_\alpha [F] = \sum\limits_{k=0}^\infty 3^{-\alpha (k+1)}\!\!\!\!\!\sum\limits_{\omega \in \AF_n}\!\!\!\!\! \Big(\big(f_{n+1}(\omega_3) - f_{n+1}(\omega_2)\big)\chi_{\omega_1} +\\ \big(f_{n+1}(\omega_1) - f_{n+1}(\omega_3)\big)\chi_{\omega_2} + \big(f_{n+1}(\omega_2) - f_{n+1}(\omega_1)\big)\chi_{\omega_3}\Big).
}
\end{Ex}

Let also~$p \in (1,\infty)$. Consider a positively~$p$-homogeneous function~$\Phi\colon \R^\ell \to \R$.
Assume~$\Phi$ is locally Lipschitz. We study the conditions on~$\Phi$ and~$\Lin$ that are necessary and sufficient for the uniform (with respect to~$F$) inequality
\eq{\label{MainIneq}
\Big|\E \Phi(\BB_\alpha[F])\Big| \lesssim \|F\|_{L_1}^p,\qquad \alpha = \frac{p-1}{p}.
}
Recall that
\eq{
\|F\|_{L_1} = \sup_n \E|F_n|.
}
The inequality~\eqref{MainIneq} requires explanations. We will study it for simple martingales only (because otherwise the correctness of the mathematical expectation on the left hand side is questionable). By a simple martingale we mean a martingale for which there exists~$N \in \N$ such that~$F_{N} = F_{N+1} = F_{N+2} = \ldots$. It is convenient to denote~$F_N$ by~$F_{\infty}$ in this case. Of course, the constant in~\eqref{MainIneq} must be independent of~$N$.

The condition~$\alpha = \frac{p-1}{p}$ in~\eqref{MainIneq} comes from homogeneity considerations. Namely, these considerations show that~$\alpha \geq \frac{p-1}{p}$ (otherwise the inequality cannot hold true). In the case~$\alpha > \frac{p}{p-1}$, a stronger inequality
\eq{
\E |\I_\alpha [F]|^p \lesssim \|F\|_{L_1}^p,\qquad \alpha > \frac{p}{p-1},
}
holds. For example, it follows from the Hardy--Littlewood--Sobolev inequality for martingales. See~\cite{Watari1964} for the dyadic version of this inequality, Theorem 5.1 in~\cite{NakaiSadasue2012} for a version for uniform filtrations, and~\cite{StolyarovYarcev2020} for a generalization to the setting of arbitrary filtrations.

The two cancellation conditions seem to be related to the problem. We start with the cancellation condition on~$\Lin$. Let~$D_1,D_2,\ldots, D_m$ be the vectors in~$V$ that correspond to martingales representing delta measures:
\eq{\label{DeltaVectors}
D_j = (\underbrace{-1,-1,\ldots, -1}_{j-1}, m-1,-1,\ldots, -1).
}
If~$v$ is a vector in~$V$ or in~$V\otimes \R^\ell$, the notation~$v_j$,~$v_j \in \R$ or~$v_j \in \R^\ell$ respectively, signifies its~$j$th coordinate (while it is quite natural for~$v$, it might need an explanation for the vector~$\Lin[D_j]$ in the formula~\eqref{WC} below). The following condition appeared in~\cite{Stolyarov2019}. 
\begin{Def}
We say that~$\Lin$ satisfies the weak cancelation condition if
\eq{\label{WC}
(\Lin[D_j])_j = 0\qquad \text{for any } j\in [1\,..\,m].
}
\end{Def}
Note that this condition cannot be fulfilled in the most pleasant case~$m=2$ (except for~$\Lin = 0$).
\begin{Def}
We say that~$\Phi$ is~$\Lin$-canceling provided 
\eq{\label{WCT}
\sum\limits_{i\ne j} \Phi\big((\Lin[D_j])_i\big) = 0\qquad \text{for any } j \in [1\,..\,m].
}
\end{Def}
We are ready to state our main result.
\begin{Th}\label{MainTh}
Let the function~$\Phi$ be locally Lipschitz and positively $p$-homogeneous. If~$T$ satisfies the weak cancelation condition and~$\Phi$ is~$\Lin$-canceling\textup, then the inequality~\eqref{MainIneq} holds true for all simple martingales with a uniform constant. 
\end{Th}
It is not clear whether the weak cancellation condition is necessary (though it appears quite naturally in the construction of the Bellman supersolution below). We will show that if it holds true, then~\eqref{WCT} is necessary for~\eqref{MainIneq} (see Corollary~\ref{NecessityOfPhiCancellation} below).
\begin{Ex}[Continuation of Example~\ref{Example1}]\label{Example2}
Note that the operator~$\Lin$ given in~\eqref{ExampleOfT} satisfies condition~\eqref{WC}. Let for simplicity~$p=2$ and~$\alpha = \frac12$. In this case\textup, the choice for a positively $2$-homogeneous function~$\Phi\colon \R \to \R$ satisfying~\eqref{WCT} is unique up to a multiplicative constant\textup:
\eq{
\Phi(t) = t|t|,\qquad t \in \R.
}
Theorem~\ref{MainTh} in this particular case says the inequality
\eq{\label{ExampleIneq}
\Big|\E \big(\BB_\frac12 [F]\big)\big| \BB_\frac12 [F]\big|\Big| \lesssim \|F\|_{L_1}^2
}
holds true for all simple martingales~$F$ with a uniform constant\textup; the operator~$\BB_{\frac12}$ is defined in~\eqref{ExampleTalpha} with~$\alpha = \frac12$.
\end{Ex}

\section{Bellman function}\label{s3}
\begin{Def}
Define the Bellman function~$\Bell_{\Lin,\Phi}\colon \R\times \R^\ell \times \R \to \R \cup \{-\infty,+\infty\}$ by the rule
\eq{\label{Bellman}
\Bell_{\Lin,\Phi}(x,y,z) = \sup\Set{\E \Phi(\BB_\alpha [F] + y)}{F_0 = x,\ \E|F_{\infty}| = z,\quad F \hbox{ is a simple martingale}}.
}
\end{Def}
We will usually suppress the indices in the notation for~$\Bell$ if this does not lead to ambiguity.
\begin{Le}
The value~$\Bell(x,y,z)$ equals~$-\infty$ if and only if
\eq{
(x,y,z) \notin \Omega = \Set{(x,y,z) \in \R\times \R^\ell \times \R}{|x| \leq z}.
}
On the boundary of~$\Omega$\textup, we have
\eq{\label{BC}
\Bell(x,y,|x|) \geq \Phi(y),\qquad x,y\in\R.
}
\end{Le}
\begin{proof}
If~$|x| > z$, there does not exist a martingale~$F$ such that
\eq{\label{BellmanPoint}
F_0 =x \quad \text{and}\quad \E|F_{\infty}| = z,
}
so in this case the supremum in~\eqref{Bellman} is taken over the empty set, and~$\Bell(x,y,z) = -\infty$. On the other hand, if~$|x| \leq z$, there exists~$F$ such that~\eqref{BellmanPoint} holds true, the martingale
\eq{
F_0 = x,\qquad F_1 = \begin{cases}
\frac{m}{2}(x+z),\quad &\text{with probability }\frac1m;\\
\frac{m}{2(m-1)}(x-z),\quad &\textup{with probability }\frac{m-1}{m},
\end{cases}\qquad F_n = F_1 \quad \text{for larger }n,
}
serves as an example. Therefore, if~$(x,y,z)\in \Omega$, then~$\Bell(x,y,z) > -\infty$. The first assertion of the lemma is proved. 

To prove~\eqref{BC}, we note that if~$|x|=z$, then there is a constant martingale satisfying~\eqref{BellmanPoint}. For this martingale,~$\BB_\alpha[F] = 0$, therefore,
\eq{
\E \Phi(\BB_\alpha [F] + y) = \Phi(y).
}
Thus, by the definition of the Bellman function,~$\Bell(x,y,|x|) \geq \Phi(y)$.
\end{proof}
The set~$\Omega$ is usually called the natural domain (or simply the domain) of~$\Bell$. Though we initially define~$\Bell$ on the whole space~$\R\times \R^\ell  \times \R$, it becomes an interesting object only on~$\Omega$. Inequality~\eqref{BC} may be thought of as boundary conditions.
\begin{Rem}\label{BellmanReformulation}
Inequality~\eqref{MainIneq} is equivalent to the following statement\textup: there exists a constant~$C > 0$ such that for every simple martingale~$F$ the inequalities
\eq{
\E \Phi(\BB_\alpha[F]) \leq C(\E|F_{\infty}|)^p\quad \text{and}\quad -\E \Phi(\BB_\alpha[F]) \leq C(\E|F_{\infty}|)^p
}
hold true. This may be restated as
\eq{
\Bell_{\Lin,\Phi}(x,0,z) \leq Cz^p\quad \text{and}\quad \Bell_{\Lin,-\Phi}(x,0,z) \leq Cz^p.
}
\end{Rem}
Recall that~$\alpha = \frac{p-1}{p}$.  
\begin{Le}
The Bellman function satisfies the {\bf main inequality}\textup: for any~$y \in \R^\ell$ and any collections of real numbers~$\{x_j\}_{j=1}^m$ and~$\{z_j\}_{j=1}^m$ such that~$|x_j| \leq z_j$ for any~$j$\textup, the inequality
\mlt{\label{MainInequality}
\Bell(x,y,z) \geq \frac{1}{m^p}\sum\limits_{j=1}^m\Bell\Big(x_j,m^\alpha y +\big(\Lin[\vec{x}]\big)_j\,, z_j\Big),\\ x=\frac{1}{m}\sum\limits_{j=1}^m x_j,\ z=\frac{1}{m}\sum\limits_{j=1}^m z_j,\ \vec{x} = (x_1-x,x_2-x,\ldots,x_m-x)\in V
}
holds true.
\end{Le}
\begin{proof}
Let~$F^1, F^2,\ldots, F^m$ be almost optimal martingales for the points~$P_j\in \Omega$, where
\eq{
P_j = (x_j,m^\alpha y +\big(\Lin[\vec{x}]\big)_j, z_j),\quad j \in [1 .. m],
}
in the sense that
\eq{\label{AlmostOptimal}
F^j_0 = x_j,\, \E |F^j_\infty| = z_j,\quad \text{and}\quad \E\Phi\Big(\BB_\alpha[F^j] + m^\alpha y + (\Lin[\vec{x}])_j\Big) \geq \Bell(P_j) -\eps.
}
Here~$\eps$ is an auxiliary small parameter. We glue the martingale~$F$ from~$F^1$,~$F^2$,\ldots,~$F^m$. More specifically, let~$w_1,w_2,\ldots,w_m$ be the elements of~$\AF_1$. We define~$F$ in the following way:
\eq{
F_0 = x, \quad F_1 = x_j\quad \text{on}\ w_j,
}
and the martingale~$F|_{w_j}$ develops like~$F^j$ in the sense that if~$w\subset w_j$ and~$w\in \AF_n$, then
\eq{
F_{n}(w) = F_{n-1}^j(w^\uparrow),
} 
here~$w^\uparrow$ is the parent of~$w$. In particular,
\eq{
f_1 = \J_w[\vec{x}], \quad w\in \AF_0.
}
Then,
\eq{
(\BB_\alpha[F])_n(w) = m^{-\alpha}(\Lin[\vec{x}])_j + m^{-\alpha}(\BB_\alpha[F^j])_{n-1}(w^\uparrow),\qquad w\in \AF_n,\ w \subset w_j.
}
Therefore,
\mlt{
\Bell(x,y,z) \geq \E \Phi(\BB_\alpha [F] + y) = \frac{1}{m}\sum\limits_{j=1}^m\E\Phi\Big(y+ m^{-\alpha}(\Lin[\vec{x}])_j + m^{-\alpha}\BB_\alpha[F^j]\Big)\stackrel{\scriptscriptstyle (1-\alpha)p=1}{=}\\
\frac{1}{m^p}\sum\limits_{j=1}^m\E\Phi\Big(m^\alpha y+(\Lin[\vec{x}])_j +\BB_\alpha[F^j]\Big) \Geqref{AlmostOptimal}\frac{1}{m^p}\sum\limits_{j=1}^m\Bell(P_j) - m^{1-p}\eps.
}
We obtain the desired inequality~\eqref{MainInequality} by choosing arbitrarily small~$\eps$.
\end{proof}
\begin{Rem}
By homogeneity of~$\Phi$\textup,
\eq{\label{Hom}
\Bell(\lambda x,\lambda y,\lambda z) = \lambda^p\Bell(x,y,z),\qquad \lambda > 0,\ (x,y,z)\in\Omega.
}
\end{Rem}
\begin{Ex}[Continuation of Example~\ref{Example2}]\label{ExactFunction}
Let all the parameters be as in Example~\ref{Example2}. Then\textup, the Bellman function is defined on the domain
\eq{
\Omega = \Set{(x,y,z) \in \R^3}{|x| \leq z}.
}
It satisfies the boundary conditions
\eq{\label{BoundIneq}
\Bell(x,y,|x|) \geq y|y|
}
and the main inequality
\eq{\label{MI}
\Bell(x,y,z) \geq \frac19\bigg(\Bell(x_1,\sqrt3 y + x_3 - x_2,z_1) + \Bell(x_2,\sqrt3 y + x_1 - x_3,z_2) + \Bell(x_3,\sqrt3 y + x_2 - x_1,z_3)\bigg),
}
whenever
\eq{
x = \frac{x_1+x_2+x_3}{3}, \quad z = \frac{z_1+z_2+z_3}{3},\quad |x_1|\leq z_1, |x_2|\leq z_2, |x_3|\leq z_3,\quad y\in\R.
}
\end{Ex}
\begin{Cor}\label{NecessityOfPhiCancellation}
Assume that~$\Lin$ is weakly canceling\textup, i.e.~\eqref{WC} holds true. Then\textup, if~\eqref{MainIneq} holds true\textup,~$\Phi$ is~$\Lin$-canceling.
\end{Cor}
\begin{proof}
Assume~\eqref{MainIneq} holds true, in particular, the values of the Bellman functions~$\Bell_{\Lin,\Phi}$ and~$\Bell_{\Lin,-\Phi}$ at the point~$(1,0,1)$ are finite (see Remark~\ref{BellmanReformulation}). We apply the main inequality to the points~$x_1=z_1 = m$ and~$x_j=y_j=0$ for~$j=2,3,\ldots,m$, and~$y=0$. In this case,~$\vec{x} = D_1$. Then,
\mlt{
\Bell_{\Lin,\Phi}(1,0,1) \geq \frac{1}{m^p}\Big(\Bell_{\Lin,\Phi}(m,0,m) + \sum\limits_{j=2}^m\Bell_{\Lin,\Phi}\big(0,(\Lin[D_1])_j,0\big)\Big) \GeqrefTwo{BC}{Hom}\\  \Bell_{\Lin,\Phi}(1,0,1) + m^{-p}\sum\limits_{j=2}^m\Phi((\Lin[D_1])_j),
}
which implies
\eq{
\sum\limits_{j=2}^m\Phi((\Lin[D_1])_j) \leq 0.
}
The same reasoning for~$\Bell_{\Lin,-\Phi}$ implies the reverse inequality. Finally, we may replace the index~$1$ with arbitrary~$j\in [1\,..\,m]$, apply the same reasoning, and obtain~\eqref{WCT}.
\end{proof}
\begin{Def}
A function~$G\colon\Omega \to \mathbb{R}$ is called a {\bf supersolution} if it satisfies the boundary condition~\eqref{BC} and the main inequality~\eqref{MainInequality}.
\end{Def}
\begin{Le}\label{SupersolutionLemma}
If~$G$ is a supersolution\textup, then~$G(x,y,z) \geq \Bell(x,y,z)$ for any~$(x,y,z)\in \Omega$. 
\end{Le}
In particular, if there exists a supersolution, then the Bellman function is finite.
\begin{proof}
We will prove that the process
\eq{\label{Supermartingale}
\bigg\{m^{-(p-1)n}G\Big(F_n, m^{\alpha n}\big(y_0 + (\BB_\alpha[F])_n\big), \E(|F_\infty|\mid \F_n)\Big)\bigg\}_n
}
is a supermartingale provided~$G$ satisfies the main inequality and~$F$ is an arbitrary simple martingale. Let us first show that this assertion implies the lemma. For that pick arbitrary~$(x,y,z)\in\Omega$ and an arbitrary simple martingale~$F$ that fulfills~\eqref{BellmanPoint}. Since the process~\eqref{Supermartingale} is a supermartingale,
\mlt{
G(x,y,z) = G\big(F_0, y, \E|F_\infty|\big) =\E G\big(F_0,y+(\BB_\alpha[F])_0, \E(|F_\infty|\mid \F_0)\big) \geq\\ m^{-(p-1)N} \E G(F_N, m^{\alpha N}(y + (\BB_\alpha[F])_N), \E(|F_\infty|\mid \F_N)) =\\ m^{-(p-1)N} \E G(F_{\infty}, m^{\alpha N}(y + \BB_\alpha[F]), |F_\infty|),
}
provided~$N$ is sufficiently large.  The latter quantity is not smaller than
\eq{
m^{-(p-1)N}\E \Phi\big(m^{\alpha N}(y+\BB_\alpha[F])\big) =\E \Phi(y+\BB_\alpha[F]),
} since we assume~$G$ satisfies the boundary condition~\eqref{BC}. Thus,
\eq{
G(x,y,z) \geq \E \Phi(y+\BB_\alpha[F])
}
for arbitrary~$F$. Taking supremum over all admissible martingales~$F$, we obtain~$G(x,y,z) \geq \Bell(x,y,z)$.

It remains to prove~\eqref{Supermartingale} is a supermartingale. Let~$w\in \AF_n$, let~$w_1,w_2,\ldots,w_m$ be its kids in~$\AF_{n+1}$. Define
\eq{
y = m^{\alpha n}\Big(y_0 + (\BB_\alpha[F])_{n}(w)\Big);\quad x_j = F_{n+1}(w_j),\ z_j = \E(|F_\infty|\mid \F_{n+1})(w_j),\qquad j \in [1\,..\,m].
}
By martingale properties,
\eq{
x = F_n(w),\ z = \E(|F_\infty|\mid F_n)(w),\quad \text{and} \ \vec{x} = \J_{w}^{-1}[f_{n+1}|_w],
}
in the sense that these quantities satisfy the same relations as they do in the main inequality~\eqref{MainInequality}. Consequently,
\eq{
G\Big(F_n, m^{\alpha n}\big(y_0 + (\BB_\alpha[F])_n\big),\E(|F_\infty|\mid \F_n)\Big)(w) = G(x,y,z)
}
and
\mlt{
G\Big(F_{n+1}, m^{\alpha(n+1)}\big(y_0 + (\BB_\alpha[F])_{n+1}\big),\E(|F_\infty|\mid \F_{n+1})\Big)(w_j) =\\ G(x_j,m^{\alpha}y + (\Lin[\vec{x}])_j,z_j),\quad j\in [1\,..\,m].
}
Thus, the inequality
\mlt{
m^{-(p-1)n} G\Big(F_n, m^{\alpha n}\big(y_0 + (\BB_\alpha[F])_n\big),\E(|F_\infty|\mid \F_n)\Big)(w) \leq \\m^{-(p-1)(n+1)}\E\Big(G\Big(F_{n+1}, m^{\alpha(n+1)}\big(y_0 + (\BB_\alpha[F])_{n+1}\big),\E(|F_\infty|\mid \F_{n+1})\Big)\,\Big|\; \F_{n}\Big)(w)
}
coincides with~\eqref{MainInequality} with~$G$ in the role of~$\Bell$. Since~\eqref{Supermartingale} is adapted to~$\F$, the latter inequality confirms~\eqref{Supermartingale} is a supermartingale.
\end{proof}
In the theorem below~$|y|$, where~$y \in \R^\ell$, means the usual Euclidean norm of a vector~$y \in \R^\ell$. 
\begin{Th}\label{Supersolution}
Let~$\Lin$ be weakly canceling and let~$\Phi$ be~$\Lin$-canceling. If~$p \leq 2$\textup, then the function
\eq{\label{SS1}
G(x,y,z) = \Phi(y) + C_1\min(|y|^{p-1}z, |y|z^{p-1}) + C_2z^p
}
is a supersolution provided~$C_1$ is sufficiently large and~$C_2$ is sufficiently large \textup(depending on~$C_1$\textup). If~$p \geq 2$\textup, then the function
\eq{\label{SS2}
G(x,y,z) = \Phi(y) + C_1 \Big(|y|^{p-1}z + |y|z^{p-1}\Big) + C_2 z^p
}
is a supersolution under the same assumption on~$C_1$ and~$C_2$.
\end{Th}
In view of Remark~\ref{BellmanReformulation} and Lemma~\ref{SupersolutionLemma}, Theorem~\ref{Supersolution} implies Theorem~\ref{MainTh}.

\section{Auxiliary functions and their properties}\label{s4}
We collect the properties of special functions used in the proof of Theorem~\ref{Supersolution}. The verification of these properties is often elementary.

Let~$\Psi\colon (\R_+)^m \to \R$ be defined by the rule
\eq{\label{Psi}
\Psi(z_1,z_2,\ldots,z_m) = \Big(\sum\limits_{j=1}^mz_j\Big)^p - \sum\limits_{j=1}^m z_j^p.
}
The function~$\Psi$ is positively $p$-homogeneous  and locally Lipschitz.
\begin{Le}\label{PsiLemma}
The function~$\Psi$ is non-negative. It vanishes only at the points~$\lambda e_j$\textup,~$\lambda > 0$ and~$j\in [1\,..\,m]$\textup, where
\eq{
e_j = (\underbrace{0,0,\ldots, 0}_{j-1}, 1, 0,\ldots, 0)
}
are the standard basic vectors. For each~$j \in [1\,..\,m]$\textup, if~$\sum_{j=1}^m z_j^p=1$ and~$\vec z$ lies in a neighborhood of~$e_j$\textup, then
\eq{\label{OriginalAssert}
\Psi(\vec z - e_j) \gtrsim |\vec z-e_j|,\qquad \vec{z} = (z_1,z_2,\ldots,z_m).
}
\end{Le}
\begin{proof}
The non-negativity of~$\Psi$ follows from H\"older's inequality (the~$\ell_1$-norm is larger than the~$\ell_p$-norm; the equality is possible only if the~$\ell_p$ norms coincide for all~$p$, which is exactly the case~$\vec{z} = \lambda e_j$). 

Let us prove the second assertion of the lemma. Assume that~$\sum_{j=1}^m z_j^p=1$ and~$\vec{z}$ lies in a sufficiently small neighborhood of~$e_1$. Note that
\eq{
\sum_{j=2}^mz_j =  |\vec{z} - e_1|_{\ell_1} - |z_1 - 1|_{\ell_1} = |\vec{z} - e_1|_{\ell_1} - \Big(1-\big(1-\sum\limits_{j=2}^m z_j^p\big)^{\frac1p}\Big) \geq |\vec{z} - e_1|_{\ell_1} - O(\sum\limits_{j=2}^m z_j^p).
}
Since~$p > 1$, this proves
\eq{\label{OneHalf}
\sum_{j=2}^mz_j \geq \frac12|\vec{z} - e_1|_{\ell_1},
}
provided the latter quantity is sufficiently small. 

Now we are ready to prove the second assertion of the lemma. We use Bernoulli's inequality:
\eq{
\big(\sum\limits_{j=1}^m z_j\big)^p \geq z_1^p + pz_1^{p-1}\sum\limits_{j=2}^mz_j = 1-\sum\limits_{j=2}^m z_j^p + p(1 + O(|\vec{z} - e_1|_{\ell_1})\sum\limits_{j=2}^mz_j.
}
Since~$p > 1$, the first sum is dominated by the last term when~$\vec{z}$ is sufficiently close to~$e_1$. So, by~\eqref{OneHalf}, we obtain
\eq{
\big(\sum\limits_{j=1}^m z_j\big)^p \geq 1 + \frac{p}{4}|\vec{z} - e_1|_{\ell_1},
}
provided~$\vec{z}$ is sufficiently close to~$e_1$. The inequality~\eqref{OriginalAssert} follows from this one since any two norms on a finite dimensional space are equivalent.
\end{proof}

The function~$\min(|y|^{p-1}z, |y|z^{p-1})$ plays an important role in Theorem~\ref{Supersolution}, so it deserves a name:
\eq{
\Min\colon \R^\ell\times \R_+\to \R,\qquad \Min(y,z) = \min(|y|^{p-1}z, |y|z^{p-1}).
}
The function~$\Min$ is positively~$p$-homogeneous. 
\begin{Le}\label{MinLip}
The function~$\Min$ is locally Lipschitz if~$p \in (1,2]$.
\end{Le}
Of course, the assertion of lemma remains true when~$p \geq 2$.
\begin{proof}
Let~$K$ be a bounded region in~$\R\times \R_+$. It suffices to prove that~$\Min$ is Lipschitz on~$K \cap \{|y| \geq z\}$ and on~$K \cap \{z \geq |y|\}$, because~$\Min$ is continuous on~$K$. It remains to notice that on each of the sets~$\{|y| \geq z\}$ and~$\{z \geq |y|\}$ the function~$\Min$ is~$C^1$-smooth.
\end{proof}
We may also write the representation
\eq{\label{MinTheta}
\Min(y,z) = |y|^{p}\theta\Big(\frac{z}{|y|}\Big) = z^{p}\theta\Big(\frac{|y|}{z}\Big),
}
where~$\theta\colon \R\to\R$ is given by formula
\eq{
\theta(t) = \min(|t|,|t|^{p-1}).
}
If~$p \leq 2$, then the function~$\theta$ is Lipschitz. Its restriction to the positive semi-axis is concave.
\begin{Le}\label{EstimateForTheta}
Let~$p \in (1,2]$. For any~$a,b\in\R$\textup, the inequality
\eq{
|\theta(a+b) -\theta(a)|\lesssim \theta(b)
}
holds true.
\end{Le}
\begin{proof}
We will consider several cases. If~$|b| \leq 1$, we use that~$\theta$ is a~$1$-Lipschitz function:
\eq{
|\theta(a+b) - \theta(a)|\leq |b| = \theta(b).
}
So, we assume~$|b| \geq 1$ in what follows. We wish to prove the inequality
\eq{
|\theta(a+b) - \theta(a)| \lesssim |b|^{p-1},\quad |b| \geq 1.
}
If either~$|a+b|$ or~$|a|$ is smaller than~$1$, then there is nothing to prove, because~$\theta(a+b) + \theta(a) \lesssim |b|^{p-1}$ in this case. Thus, we may also assume~$|a| \geq 1$,~$|a+b| \geq 1$ and try to prove the inequality
\eq{
||a+b|^{p-1} - |a|^{p-1}|\lesssim b^{p-1}
}
in this regime. If~$|a| \leq 2|b|$, then we may use the estimate
\eq{
||a+b|^{p-1} - |a|^{p-1}| \leq (3^{p-1}+2^{p-1})|b|^{p-1}.
}
In the other case, we use the differentiability of~$(1+t)^{p-1}$ at zero and write
\eq{
||a+b|^{p-1} - |a|^{p-1}| \lesssim |a|^{p-1}\frac{|b|}{|a|} \lesssim |b|^{p-1}
}
since~$p \leq 2$ and~$|b| \leq \frac{|a|}{2}$.
\end{proof}
The following lemma is simple, so we omit the proof.
\begin{Le}\label{EpsilonLemma}
Let~$p \geq 1$. For any~$\eps > 0$ there exists a constant~$C_\eps$ such that the inequality
\eq{
|a+b|^{p-1} \leq (1+\eps)|a|^{p-1} + C_\eps|b|^{p-1}
}
holds true for any~$a,b \in \R$.
\end{Le}
\section{Verification of the main inequality}\label{S5}
Throughout this section let~$y, x, z$ and~$\{x_j\}_{j=1}^m$,~$\{z_j\}_{j=1}^m$ be the same as in the main inequality~\eqref{MainInequality}. We treat the main inequality for separate summands in formulas~\eqref{SS1} and~\eqref{SS2} individually. With the function~$z^p$, this is simple:
\eq{\label{z^p}
z^p - \frac{1}{m^p}\sum\limits_{j=1}^mz_j^p = m^{-p}\Psi(\vec{z}),
}
where the function~$\Psi$ is defined in~\eqref{Psi}, and we use our standard notation~$\vec{z} = (z_1,z_2,\ldots,z_m)$. The function~$\Min$ requires a more serious study.
\begin{Le}\label{MinLemma}
Let~$p \in (1,2]$ and let~$\Lin$ be weakly canceling in the sense that~\eqref{WC} holds true. There exists a constant~$c > 0$ such that
\eq{
\Min(y,z) - \frac{1}{m^p}\sum\limits_{j=1}^m \Min(m^\alpha y+(\Lin[\vec x])_j, z_j) \geq c\Min(y,z) + O(\Psi(\vec{z})).
}
\end{Le}
\begin{proof}
We derive from~\eqref{MinTheta} and Lemma~\ref{EstimateForTheta} that
\eq{
\Big|\Min(m^\alpha y+(\Lin[\vec x])_j, z_j) - \Min(m^\alpha y,z_j)\Big| \lesssim \Min((\Lin[\vec x])_j, z_j)
}
for any~$j$. Thus,
\mlt{
\Min(y,z) - \frac{1}{m^p}\sum\limits_{j=1}^m \Min(m^\alpha y+(\Lin[\vec x])_j, z_j) \geq\\ \underbrace{\Min(y,z) - \frac{1}{m^p}\sum\limits_{j=1}^m \Min(m^\alpha y, z_j)}_{I_1} + O\Big(\underbrace{\sum\limits_{j=1}^m \Min((\Lin[\vec x])_j, z_j)}_{I_2}\Big).
}
Let us estimate~$I_1$ and~$I_2$ separately. We start with~$I_1$:
\mlt{
I_1 = \Min(y,z) - m^{-p+\alpha p}\sum\limits_{j=1}^m\Min(y,m^{-\alpha} z_j) = |y|^p\Big[\theta\Big(\frac{z}{|y|}\Big) - m^{-1}\sum\limits_{j=1}^m\theta\Big(\frac{z_j}{m^{\alpha}|y|}\Big)\Big]  \Gref{\scriptscriptstyle \theta\, \text{\tiny is concave}}\\
|y|^p\Big[\theta\Big(\frac{z}{|y|}\Big) - \theta\Big(\frac{z}{m^\alpha|y|}\Big)\Big] \geq c\Min(|y|,z),
}
since~$\theta(\lambda t) \leq \lambda^{p-1}\theta(t)$ for~$\lambda < 1$.

We claim the following estimate for~$I_2$:
\eq{\label{I2estimate}
\sum\limits_{j=1}^m \Min((\Lin[\vec x])_j, z_j) \lesssim \Psi(\vec{z}).
}
Both the left hand and the right hand sides are positively~$p$-homogeneous, so, we may assume~$\sum_{j=1}^mz_j^p = 1$ without loss of generality. By Lemma~\ref{PsiLemma}, it suffices to show that the expression on the left hand side is a Lipschitz function (under our assumption about the~$z_j$) that vanishes whenever~$\vec{z} = e_j$ for some~$j \in [1\,..\, m]$. The first assertion follows from Lemma~\ref{MinLip}. To prove the second assertion, we note that since~$(x_j,y,z_j)\in \Omega$ for every~$j$, the equality~$\vec{z} = e_j$ implies~$\vec{x} = \lambda D_j$ (the~$D_j$ are defined in~\eqref{DeltaVectors}) for some~$\lambda \leq \frac{1}{m}$. In this case, the left hand side vanishes since~$\Lin$ is weakly canceling.
\end{proof}
\begin{Ex}[Continuation of Example~\ref{Example2}]
In this case\textup,~\eqref{I2estimate} is reduced to a simple consequence of the triangle inequality
\eq{
|x_3 - x_2|z_1 + |x_1- x_3|z_2 + |x_2 - x_1|z_3 \lesssim z_1z_2 + z_2z_3 + z_1z_3.
}
\end{Ex}

\begin{Le}\label{SumLemma}
Let~$p \in [2,\infty)$ and let~$\Lin$ be canceling. There exists a constant~$c > 0$ such that
\alg{
\label{p-11}|y|^{p-1}z - \frac{1}{m^p}\sum\limits_{j=1}^m \Big|m^\alpha y+(\Lin[\vec x])_j\Big|^{p-1} z_j \geq c|y|^{p-1}z + O(\Psi(\vec{z}));\\
\label{1p-1}|y|z^{p-1} - \frac{1}{m^p}\sum\limits_{j=1}^m \Big|m^\alpha y+(\Lin[\vec x])_j\Big| z_j^{p-1} \geq c|y|z^{p-1} + O(\Psi(\vec{z})).
}
\end{Le}
\begin{proof}
The reasoning here is similar to the proof of Lemma~\ref{MinLemma}, however, some details differ. Let us prove~\eqref{p-11} first. We pick some tiny~$\eps$ and apply Lemma~\ref{EpsilonLemma}:
\mlt{
|y|^{p-1}z - \frac{1}{m^p}\sum\limits_{j=1}^m \Big|m^\alpha y+(\Lin[\vec x])_j\Big|^{p-1} z_j \geq\\ |y|^{p-1}z - (1+\eps)\frac{1}{m^p}\sum\limits_{j=1}^m m^{\alpha(p-1)} |y|^{p-1} z_j + O\Big(\sum\limits_{j=1}^m \Big|(\Lin[\vec x])_j\Big|^{p-1} z_j\Big),
}
where the constant in~$O$ depends on~$\eps$. Similar to the proof of Lemma~\ref{MinLemma}, the second summand is dominated by~$\Psi(\vec{z})$ (it is important that~$p \geq 2$ here since otherwise the function in question is not locally Lipschitz). The first summand equals
\eq{
\Big(1 - (1+\eps)m^{1-p + \alpha (p-1)}\Big)|y|^{p-1}z,
}
and the coefficient in the parenthesis is non-negative provided~$\eps$ is sufficiently small; recall~$\alpha \in (0,1)$.

Now let us prove~\eqref{1p-1}. We use the triangle inequality:
\mlt{
|y|z^{p-1} - \frac{1}{m^p}\sum\limits_{j=1}^m \Big|m^\alpha y+(\Lin[\vec x])_j\Big| z_j^{p-1} \geq |y|z^{p-1} -m^{\alpha - p}\sum\limits_{j=1}^m | y| z_j^{p-1} - \frac{1}{m^p}\sum\limits_{j=1}^m\big|(\Lin[\vec{x}])_j\big|z_j^{p-1}.
}
The last summand is dominated by~$\Psi(\vec{z})$. As for the first two summands, we use the inequality
\eq{
\Big(\sum\limits_{j=1}^m z_j \Big)^{p-1} \geq \sum\limits_{j=1}^m z_j^{p-1},\qquad p \geq 2,
}
and obtain:
\eq{
|y|z^{p-1} -m^{\alpha - p}\sum\limits_{j=1}^m | y| z_j^{p-1} \geq (m^{1-p} - m^{\alpha - p})|y|\sum\limits_{j=1}^m z_j^{p-1} \gtrsim |y|z^{p-1}
}
since~$\alpha < 1$.
\end{proof}

\begin{Le}\label{PhiLemma}
Let~$p\in [1,2]$. If~$\Phi$ is~$\Lin$-canceling in the sense that~\eqref{WCT} holds true\textup, then
\eq{\label{DifPhi}
\Big|\Phi(y) - \frac{1}{m^p}\sum\limits_{j=1}^m\Phi(m^\alpha y + (\Lin[\vec x])_j)\Big| \lesssim \Min(|y|,z) + \Psi(\vec{z}).
}
\end{Le}
\begin{proof}
Due to homogeneity, we may assume~$|y| = 1$. Consider two cases.

\paragraph{Case~$|\vec{x}| \leq 1$.} We use that~$\Phi$ is locally Lipschitz:
\eq{
\Big|\Phi(y) - \frac{1}{m^p}\sum\limits_{j=1}^m\Phi(m^\alpha y + (\Lin[\vec x])_j)\Big| \lesssim \sum\limits_{j=1}^m |(\Lin[\vec{x}])_j| \lesssim \min(1,z) \leq \Min(1,z).
}

\paragraph{Case~$|\vec{x}| \geq 1$.} In this case~$z \gtrsim 1$ and~$\Min(1,z)\gtrsim 1$. Therefore, the left hand side of~\eqref{DifPhi} is bounded by
\eq{
1+ \sum\limits_{j=1}^m |(\Lin[\vec{x}])_j|^{p-1} + \Big|\sum\limits_{j=1}^m \Phi\big((\Lin[\vec x])_j\big)\Big| \lesssim \Min(1,z) + \Big|\sum\limits_{j=1}^m \Phi\big((\Lin[\vec x])_j\big)\Big|.
}
It remains to prove the inequality
\eq{\label{SuperMain}
\Big|\sum\limits_{j=1}^m \Phi\big((\Lin[\vec x])_j\big)\Big| \lesssim \Psi(\vec z).
}
Due to homogeneity, we may assume~$\sum_{j=1}^m z_j^p = 1$. In view of Lemma~\ref{PsiLemma}, it suffices to show that the function on the left is Lipschitz and vanishes in the cases where~$\vec{z} = e_j$ for some~$j\in [1\,..\, m]$. The first assertion follows from our assumption that~$\Phi$ is locally Lipschitz. Let us verify the second assertion. If~$\vec{z} = e_j$, then, since~$(x_j,y,z_j) \in\Omega$,~$\vec{x} = \lambda D_j$ for some~$\lambda \leq \frac{1}{m}$, and the left-hand side vanishes by the condition that~$\Phi$ is~$\Lin$-canceling. 
\end{proof}
\begin{Rem}\label{PhiRem}
In the case~$p \in[2,\infty)$\textup, the estimate
\eq{\label{DifPhi2}
\Big|\Phi(y) - \frac{1}{m^p}\sum\limits_{j=1}^m\Phi(m^\alpha y + (\Lin[\vec x])_j)\Big| \lesssim |y|z^{p-1} + |y|^{p-1}z + \Psi(\vec{z})
}
holds true. The proof is completely similar.
\begin{Ex}[Continuation of Example~\ref{Example2}]
In this case\textup,~\eqref{SuperMain} is reduced to a 
\eq{
\Big|(x_3-x_2)|x_3 - x_2| + (x_1 - x_3)|x_1- x_3| + (x_2 - x_1)|x_2 - x_1|\Big| \lesssim z_1z_2 + z_2z_3 + z_1z_3.
}
One\footnote{I am grateful to Leonid Slavin for this proof.} may prove this in an elementary way\textup:
\mlt{
\Big|(x_3-x_2)|x_3 - x_2| + (x_1 - x_3)|x_1- x_3| + (x_2 - x_1)|x_2 - x_1|\Big| \leq\\ z_1\Big||x_1- x_3| - |x_2 - x_1|\Big| + z_2\Big||x_2- x_1| - |x_3 - x_2|\Big| + z_3 \Big||x_3- x_2| - |x_1 - x_3|\Big| \leq \\
z_1(z_2 + z_3) + z_2(z_1+z_3) + z_3(z_1+z_2) \leq 2\big(z_1z_2 + z_1z_3 + z_2z_3\big).
}
\end{Ex}

\end{Rem}
\begin{proof}[Proof of Theorem~\ref{Supersolution}] The case~$p \leq 2$ follows from~\eqref{z^p} and Lemmas~\ref{MinLemma}, \ref{PhiLemma}. The case~$p \geq 2$ follows from~\eqref{z^p}, Lemma~\ref{SumLemma}, and Remark~\ref{PhiRem}.
\end{proof}

\section{Conjecture and open question}\label{s6}
The martingale considerations hint us that Maz'ya's conjecture mentioned at the very beginning of the paper may hold in a more general form. Let now~$\Lin$ be a mapping of~$S^{d-1}$ to~$\R^\ell$ for some~$\ell \in \N$. Assume~$\Lin$ is sufficiently smooth, at least, H\"older continuous. Let this mapping be weakly cancelling in the sense that
\eq{\label{WCF}
\int\limits_{S^{d-1}} \Lin(\zeta)\,d\sigma(\zeta) = 0.
}
See~\cite{Stolyarov2019} and~\cite{Stolyarov2021} how this matches the martingale weak cancelation condition~\eqref{WC}. Let~$\alpha \in (0,d)$. Consider the operator
\eq{
\BB_\alpha[f] = \mathcal{F}^{-1}\Big(\frac{\Lin(\xi/|\xi|) \hat{f}(\xi)}{|\xi|^\alpha}\Big),\qquad f\in L_1(\R^d).
}
The hat symbol and~$\mathcal{F}$ denote the Fourier transform. Note that~$\BB_\alpha$ is an~$-\alpha$-homogeneous Fourier multiplier, so it is a convolutional operator whose kernel is homogeneous of order~$\alpha - d$. Define the function~$\tilde{\Lin}\colon S^{d-1} \to \R^\ell$ by the rule
\eq{
\frac{\tilde{\Lin}(x)}{|x|^{d-\alpha}} = \mathcal{F}^{-1}\Big(\frac{\Lin(\xi/|\xi|)}{|\xi|^\alpha}\Big)(x).
}
Let~$p$ satisfy the natural homogeneity relation
\eq{
\frac{p-1}{p} = \frac{\alpha}{d}
}
and let~$\Phi \colon \R^\ell \to \R$ be a positively~$p$-homogeneous locally Lipschitz function. 
\begin{Conj}
The inequality
\eq{
\Big|\int\limits_{\R^d} \Phi(\BB_\alpha[f](x))\,dx\Big| \lesssim \|f\|_{L_1}^p
}
holds true for all compactly supported functions~$f$ such that~$\int f = 0$\textup, with a uniform constant\textup, if and only if~$\Phi$ is~$\Lin$-canceling in the sense that
\eq{\label{PhiLin}
\int\limits_{S^{d-1}}\Phi(\tilde{\Lin}(x))\,d\sigma(x) = 0.
}
\end{Conj}
The original Maz'ya's conjecture may be obtained by choosing~$\alpha = 1$,~$\ell = d$, and~$\Lin(\zeta) = \zeta$. By spherical symmetry,~$\tilde{\Lin}(x) = cx$ for some constant~$c$ and~\eqref{PhiLin} reduces to~\eqref{MazjaCancellation} in this case. It is not clear whether~\eqref{WCF} is necessary in this context or not.

We end the paper with an open question that is interesting from the Bellman function point of view.
\begin{Que}
What is the exact expression for the Bellman function described in Example~\textup{\ref{ExactFunction}?} In other words, can one compute the function~$\Bell \colon\Omega \to \R$ that is minimal among all function satisfying the boundary inequality~\eqref{BoundIneq} and the main inequality~\eqref{MI}\textup{?} Maybe\textup{,} there is another way to find the sharp constant in the inequality~\eqref{ExampleIneq}\textup{?}
\end{Que}

\bibliography{mybib}{}
\bibliographystyle{amsplain}

St. Petersburg State University, Department of Mathematics and Computer Science;

St. Petersburg Department of Steklov Mathematical Institute;

d.m.stolyarov at spbu dot ru.
\end{document}